\documentclass[12pt,twoside,a4paper]{article}
\ifx\pdfpageheight\undefined\PassOptionsToPackage{dvips}{graphicx}%
\else%
\PassOptionsToPackage{pdftex}{graphicx}
\PassOptionsToPackage{pdftex}{color}
\fi
\usepackage{amsfonts}
\usepackage[francais]{babel}
\usepackage{amssymb}
\usepackage{graphicx}
\usepackage[utf8]{inputenc} 
\usepackage[T1]{fontenc}

\ifx\pdfpageheight\undefined
\else
\usepackage[bookmarksopen=false,pdftex=true,breaklinks=true,%
     backref=page,pagebackref=true,plainpages=false,%
     hyperindex=true,pdfstartview=FitH,%
     pdfpagelabels=true,linkcolor=blue,citecolor=red,
     colorlinks=false]%
	{hyperref}
\fi

\newtheorem{theorem}{Théorème}   

\newtheorem{proposition}{Proposition}[section]

\newtheorem{definition}[proposition]{Définition}

\newcommand {\junk}[1]{}

\newenvironment{proof}[1]{
\trivlist \item[\hskip \labelsep{\bf #1}]}{\hfill\mbox{$\Box$}
\endtrivlist}
\makeatletter
\def\maketitle{\par
\begingroup
\def\@makefnmark{\hbox
to 0pt{$^{\@thefnmark}$\hss}}
\if@twocolumn
\twocolumn[\@maketitle]
\else \newpage
\global\@topnum\z@ \@maketitle
\fi\thispagestyle{plain}\@thanks
\endgroup
\setcounter{footnote}{0}
\let\maketitle\relax
\let\@maketitle\relax
\gdef\@thanks{}\gdef\@author{}\gdef\@title{}\let\thanks\relax} 
\makeatother

\def\.@{\char'76}

\def \Z{\mathbb{Z}}
\def \N{\mathbb{N}} 

\def\={\hbox{~{\bf =}~}}

\def \num{{\rm n}$^{\rm o}$}
 
\def \noi {\noindent}
\def \ss {\smallskip}
\def \sni {\ss\noi}
\def \ms {\medskip}

\def \bs {\bigskip}
\def \bni {\bs\noi}
\def \snic#1 {\vspace{.1cm}\noindent\centerline{$#1$}\vspace{.1cm}}
\def \snif#1#2#3 {\vspace{#1}\noindent\centerline{$#3$}\vspace{#2}}

\def \Av {{\rm Av}}

\def \H{{\rm H}}
\def \KH{{K^\H}}
\def \VH{{V^\H}}
\def \Kac{{K^{\rm ac}}}
\def \Vac{{V^{\rm ac}}}
\def \Ksep{{K^{\rm sep}}}
\def \Vsep{{V^{\rm sep}}}

\def \GK{\Gamma_K}
\def \GL{\Gamma_L}
\def \GKac{\Gamma_{\Kac}}
\def \KR{\overline{K}}
\def \LR{\overline{L}}
\def \GtK{\widetilde{\GK}}
\def \GtKa{\widetilde{\GKac}}

\def \Tac{{\bf T_{cvac}}}
\def \ba{{\bf a}}

\def \num {{n$^{{\rm o}}$}}

\def \pgn {poly\-gone de Newton }

\def \dimm {description immédiate }

\def \iso {iso\-mor\-phisme }

\RequirePackage{etoolbox}

\patchcmd{\sectionmark}{\MakeUppercase}{}{}{}

\begin{document}
\title{ Construction du hensélisé d'un corps valué }
\author{
F.-V. Kuhlmann
($\!$\thanks{Department of Mathematics and Statistics
     University of Saskatchewan
     106 Wiggins Road, Saskatoon, SK, S7N 5E6, Canada.
E-mail: fvk@math.usask.ca.
     Home Page: \url{http://math.usask.ca/~fvk/index.html}}~),
Henri Lombardi
($\!$\thanks{Laboratoire de Mathématiques, 
(UMR CNRS 6623), UFR des Sciences et
Techniques,
Université de Franche-Comté, 25030 Besan\c{c}on cedex, France.
E-mail: henri.lombardi@univ-fcomte.fr. Home Page: \url{http://hlombardi.free.fr}}~)
}
\date{Mars 2000}
\thispagestyle{empty}
\maketitle

\begin{quotation}
Ceci est un article paru au Journal of Algebra,
 {\bf 228}, (2000), p. 624--632. Nous avons rectifié l'orthographe et mis à jour les références. Nous avons aussi rectifié un typo qui s'est glissé pour une cause inconnue dans l'article paru au Journal of Algebra concernant les pentes du polygone de Newton. 
\end{quotation}

\begin{abstract}
We give an explicit construction of the henselization of a valued field, with a constructive proof. It is analogous to the construction of the real closure of a discrete ordered field.

Nous donnons une construction explicite, et constructivement 
prouvée, du hensélisé d'un corps valué. Cette construction 
peut être considérée comme l'analogue, dans le cas valué, de 
la construction de la clôture réelle d'un corps ordonné. 
\end{abstract}

\bni Classification AMS : 12J10, 12F05, 13J15, 12Y05, 03F65

\bni Mots clés :  Corps valué, Hensélisation, Mathématiques 
constructives.
\section*{Introduction} \label{sec Introduction}
\addcontentsline{toc}{section}{Introduction}

Il est connu que le hensélisé d'un corps valué est unique à 
\iso unique près. Cette situation est analogue à celle de la 
clôture réelle d'un corps ordonné. Dans ce dernier cas, non 
seulement on sait calculer dans la clôture réelle en utilisant des 
algorithmes dérivés de l'algorithme de Sturm, mais en plus on peut 
{\em  prouver constructivement} que ces calculs forment un tout 
cohérent, ce qui donne en fin de compte une construction 
entièrement satisfaisante de la clôture réelle (cf.  
\cite{Del,Hol,Lom1,LR,Sand}). 

\ss Il est naturel que le hensélisé d'un corps valué puisse lui 
aussi être construit de manière explicite. Du point de vue 
technique, les algorithmes du type Sturm basés sur des divisions 
successives sont remplacés, en théorie des corps valués, par 
l'usage systématique des polygones de Newton. 

\ss Dans cette note, nous utilisons des résultats de l'article 
\cite{CLR}  pour certifier constructivement la construction 
que nous donnons du hensélisé d'un corps valué. 

Signalons également un article en
préparation \cite{KL} concernant les calculs dans les extensions 
algébriques arbitraires d'un corps valué.
\section{Extensions algébriques d'un corps valué} 
\label{sec ext alg cov} 

Nous précisons tout d'abord quelques notations. Un corps valué est 
donné par un couple $(K,V)$ où $K$ est un corps et $V$ un anneau 
de valuation de ce corps, c.-à-d. un anneau qui vérifie l'axiome 
de Krull 
$$
\forall x\in K^{\times}~~~(~x\in V~ {\rm  ou} ~ 1/x\in V~)
$$
On notera $v_K:K^{\times}\rightarrow K^{\times}/V^{\times}$ la 
valuation correspondante. On étend $v_K$ à $K$ tout entier en 
posant $v_K(0)=\infty$. Les non unités de $V$, qui sont aussi les 
$x\in K$ tels que $v_K(x)>v_K(1)$, forment son idéal maximal que 
nous noterons $M_V$. Le groupe ordonné $\GK:=K^{\times}/V^{\times}$ 
est appelé le groupe de valuation, et le corps $\KR:=V/M_V$ le corps 
résiduel du corps valué $(K,V)$. Nous notons 
$\GtK:=\GK\cup\{\infty\}$. Lorsque 
$(L,W)$ est une extension valuée de $(K,V)$ (i.e., $K\cap W=V$), 
$\GK$ s'identifie à un sous-goupe de $\GL$ et $\KR$ à un sous 
corps de $\LR$. Si en outre $L$ est une clôture algébrique de $K$ 
alors $\GL$ s'identifie à la clôture divisible de $\GK$. 

\ss D'un point de vue calculatoire, pour traiter un corps valué 
$(K,V)$, on peut partir d'un anneau intègre $A$, dont le corps des 
fractions est $K$, dans lequel les opérations de base de l'anneau 
sont explicites, et à l'intérieur duquel on dispose des 
prédicats $~v_K(x)\geq v_K(y)~$ et $~x=0~$.  

Les algorithmes utilisés dans nos constructions prennent en entrée 
une liste $\ba=(a_1,\ldots,a_n)$ d'éléments de $A$, utilisent des 
polynômes $P_1,\ldots,P_r$ de $\Z[x_1,\ldots,x_n]$ et selon la 
réponse à certains tests

\snic{P_i(\ba)=0~?~~~~~~{\rm  ou}~~~~~~~ v_K(P_h(\ba))\geq 
v_K(P_k(\ba))~?
}

\sni produisent une sortie convenable. Nous appellerons de tels 
algorithmes des algorithmes {\em  uniformes dans $(K,V)$} et les 
calculs correspondants des {\em  calculs  uniformes dans $(K,V)$}.  

\ss Un théorème classique affirme que tout corps valué $(K,V)$ 
peut être étendu en un corps valué algébriquement clos 
$(\Kac,\Vac)$ ($\Vac\cap K = V$). Ce théorème ne peut pas être 
prouvé constructivement (en particulier il n'existe pas de méthode 
constructive générale pour obtenir une clôture algébrique d'un 
corps.) 

Il existe néanmoins des versions constructives de ce théorème 
classique. En voici une. 

Notons $\Tac$ la théorie formelle des corps valués 
algébriquement clos basée sur le langage des anneaux, avec en 
outre un prédicat $\Av(x)$ pour l'appartenance à l'anneau de 
valuation. Si   $(K,V)$ est un corps valué, on considère la 
théorie formelle $\Tac(K,V)$ qui est la théorie formelle 
construite à partir de 
$\Tac$ dans laquelle on a introduit comme constantes et relations le 
diagramme de $(K,V)$. 

Le théorème classique que nous venons d'évoquer est 
équivalent, modulo le théorème de complétude de Gödel (qui 
est une variante faible de l'axiome du choix{\footnote{Il est 
surprenant que dans presque tous les livres de logique mathématique, 
le théorème de complétude de Gödel soit affirmé comme une 
vérité de caratère absolu, avec le même degré de certitude 
par exemple que le théorème d'incomplétude, ou que le 
théorème sur la transcendance de $\pi$. Alors que le théorème 
de complétude n'est, du point de vue classique lui-même, qu'une 
variante faible de l'axiome du choix. Ce contre-sens manifeste commis 
dans les livres de logique (par les professionnels de la chose, donc) 
est sans doute d\^u au grand réconfort moral que le théorème de 
complétude apporte aux adeptes de la théorie des ensembles 
classique~: notre logique est la bonne puisque le théorème de 
complétude dit qu'elle est adéquate à la théorie naïve 
des ensembles. Eh non, c'est d'une théorie (pas si naïve que 
cela) des ensembles avec axiome du choix et principe du tiers exclu 
qu'il s'agit.}}), au théorème concret suivant~: 

\begin{theorem} 
\label{th Tac}
Avec les notations ci-dessus, la théorie formelle $\Tac(K,V)$ est 
cohérente.
\end{theorem}

Une preuve constructive de ce théorème concret peut être 
trouvée par exemple dans \cite{CLR} (en mettant bout à bout les 
théorèmes 1.1 et 4.3).

\ss Rappelons maintenant quelques résultats de la théorie 
classique (a priori non constructive) du hensélisé d'un corps 
valué. \`A l'intérieur d'une clôture algébrique valuée 
$(\Kac,\Vac)$ de $(K,V)$ se trouve le hensélisé $(\KH,\VH)$  de 
$(K,V)$ (avec $\VH=\Vac\cap \KH$). Ce corps est défini comme le 
corps fixe du groupe des $K$--automorphismes de $\Ksep$ (la clôture 
séparable de $K$) qui fixent $\Vsep=\Vac\cap \Ksep$. Le corps $\KH$ 
peut être caractérisé de manière plus terre à terre comme la 
plus petite sous extension de $(\Kac,\Vac)$ qui soit stable par le 
processus d'ajout d'un {\em  zéro à la Hensel}. Un zéro à la 
Hensel codé par un couple $(P,a)\in V[X]\times V$,  sous la 
condition que $P(a)\in M_V$ et $P'(a)\in V^{\times}$, est l'unique 
élément $\alpha\in \Kac$ qui vérifie 
$$P(\alpha)=0~~~~~~{\rm  et}~~~~~~ \alpha - a\in M_V
$$

L'extension $(\KH,\VH)$ est immédiate en ce sens que ni le groupe de 
valuation, ni le corps résiduel ne changent. Tout élément de 
$\KH$ possède alors une {\em  \dimm} dans $K$ au sens suivant.

\begin{definition} \label{def explicitly immediate}
Soit $\xi$ un élément de $\Kac$. Soit $(L,W):=(K [\xi],\Vac\cap 
L)$. Alors $\xi$ est dit {\em  immédiat au dessus de}  $(K,V)$ s'il 
existe un $x\in K$ et  $\mu\in M_W$ avec $\xi = x (1+\mu)$.
Nous disons que $x$ est une {\em description immédiate} de $\xi$ 
dans $(K,V)$.
\end{definition}
\section{Calculs dans le hensélisé} 
\label{sec calc hens}
Dans cette section, nous faisons comme si la théorie classique du 
hensélisé d'un corps valué assurait l'existence de $(\Kac,\Vac)$ 
et de $(\KH,\VH)$ et nous montrons comment il est alors possible de 
calculer au sein du hensélisé. 

Puisque $\GKac$ s'identifie à la clôture divisible de $\GK$, on peut 
considérer la valuation $v_\Kac$ comme un prolongement de $v_K$, et 
nous les noterons toutes deux, pour simplifier, tout simplement $v$.

\smallskip Un élément du hensélisé est obtenu par extensions 
successives de $(K,V)$, en ajoutant des zéros à la Hensel l'un 
après l'autre. Lors de la première étape, $\xi_1$ est un zéro 
à la Hensel codé par $(P_1,a_1)\in V[X]\times V$. Ceci fournit le 
corps valué $(K_1,V_1)$ avec $K_1=K[\xi_1]$ et $V_1= K_1\cap\Vac$. 
Un élément arbitraire de $K_1$ est donné sous la forme 
$R(\xi_1)$ avec $R\in K[X]$. On poursuit ensuite de manière 
récursive. Pour le zéro $\xi_2$,  $(K_1,V_1)$ remplace $(K,V)$ et 
ainsi de suite. 

Si nous savons dans chacune de ces extensions successives, calculer 
pour n'importe quel élément $\xi$ une \dimm $x(\xi)$ dans $K$, 
alors nous savons donner une réponse aux tests $~\xi=0~?~$ et 
$~v(\xi)\ge v(\xi')~?~$ (il suffit de faire le test sur des 
descriptions immédiates de~$\xi$ et $\xi'$).  D'autre part, si $\xi$ 
est obtenu comme élément d'une extension $K[\xi_1,\ldots,\xi_n]$  
et  $\zeta$ est obtenu comme élément d'une extension 
$K[\zeta_1,\ldots,\zeta_m]$, alors on peut les considérer tous deux 
comme éléments de $K[\xi_1,\ldots,\xi_n,\zeta_1,\ldots,\zeta_m]$, 
de sorte que l'on a bien la possibilité de faire tous les calculs et 
tous les tests correspondants à la structure de corps valué 
$(\KH,\VH)$.

\ms Rappelons que le \pgn d'un polynôme $P(X)=\sum_{i=0,\ldots,d} 
p_iX^i\in K[X]$ (avec $p_d\not= 0$) est formé à partir de la liste 
de couples dans $\N\times \GtK$ 

\snic{[(0,v(p_0)),(1,v(p_1)),\ldots,(d,v(p_d))].
}

Le \pgn est ``l'enveloppe convexe inférieure'' de cette liste. 
Formellement on peut le définir comme la liste extraite qui 
vérifie la condition suivante~: les couples $(i,v(p_i))$ et 
$(j,v(p_j))$ sont consécutifs dans la liste extraite si et seulement 
si on a~:

\centerline {si  $ 0\leq k < i~~~~~~~~ (v(p_i)-v(p_k))/(i-k) ~<~  
(v(p_i)-v(p_j))/(i-j)$} 

\centerline {si $ i < k < j~~~~~~~~ (v(p_i)-v(p_k))/(i-k) ~\ge~  
(v(p_i)-v(p_j))/(i-j)$}

\centerline {~si $ j < k \le d~~~~~~~~ (v(p_i)-v(p_j))/(i-j) ~<~ 
(v(p_j)-v(p_k))/(j-k) $} 

\noi On montre facilement que si  $(i,v(p_i))$ et $(j,v(p_j))$ sont 
deux sommets consécutifs du \pgn du polynôme $P$, celui-ci admet 
exactement $~j-i~$ zéros dans $\Kac$ dont la valuation soit égale 
à 
$~(v(p_i)-v(p_j))/(j-i)~$. En particulier, si $~j=i+1$, le zéro 
correspondant est nécessairement dans le hensélisé $\KH$. La 
proposition suivante, dont la preuve ressort d'un calcul immédiat, 
donne quelques précisions sur ce cas. Remarquons aussi que, vues nos 
hypothèses sur $(K,V)$, nous pouvons déduire du \pgn de $P$ la 
liste ordonnée des valuations des zéros de $P$ dans 
$\GtKa=\GKac\cup\{\infty\}$. Notons aussi que si $0$ est un zéro de 
$P$ de multiplicité $h$, cela correspond à une pente infinie du 
segment initial de largeur $h$ dans le polygone de Newton.

Introduisons une terminologie~: 

\begin{definition} 
\label{def poly spe}
Un polynôme $T\in V[X]$ est dit {\em  spécial} s'il est de la forme 
$~X^d-X^{d-1}+t_{d-2}X^{d-2}+\cdots +t_1X+t_0~$ avec les $~t_i\in 
M_V$.  Le zéro de $T$ codé à la Hensel par $(T,1)$ est appelé 
le {\em  zéro spécial} de $T$.
\end{definition}

Notez que les autres zéros de $T$ dans $\Kac$ sont dans $M_\Vac$.

\begin{proposition} \label{lem NPH} 
{\rm  (Lemme de Hensel, version polygone de Newton)}
 Soit $(K,V)$ un corps valué et $P\in V [X]$ un polynôme~: 
$P(X)=\sum_{i=0,\ldots,d} p_iX^i$. Supposons que le \pgn de $P$ 
admette une pente ``isolée'' de $(k,v(p_k))$ à $(k+1,v(p_{k+1}))$. 

\sni 1) Alors l'unique zéro  $\alpha $ de $P$  tel que 
$v(\alpha)=v(p_k)-v(p_{k+1})$ admet $~-  p_k  /  p_{k+1}~$ pour \dimm: 

\snic{\alpha = - { p_k  \over  p_{k+1} }(1+\mu) \quad {\rm avec } \; 
\mu \in M_\Vac
}

\sni 2) En outre $\nu=1+\mu$  est un zéro à la Hensel codé par 
$(Q,1)$ où  $Q\in  V [Y]$ est donné par  

\snic{Q(Y)={p_{k+1}^k \over p_k^{k+1}}P(-{p_k\over p_{k+1}}Y)
}

\noi ($\nu$ est l'unique zéro de $Q$ dans $\Vac^{\times}$). 

\sni 3) Si on pose $~\sum_{i=0,\ldots,d} r_iX^i:=R(X):=Q(1+X)~$, on a 
$~v(r_0)>0~$ et  $~v(r_1)=0~$. \\
Si $r_0=0$ alors $\mu=0$. Si $r_0\neq 
0$ on pose $~S(X):=(1/r_0)R(-r_0X/r_1)~$ et $~T(X)=X^dS(1/X)~$ et on 
obtient que $T\in V[X]$ est un polynôme spécial, de sorte que tous 
les zéros de $T$, sauf celui correspondant au zéro $\alpha$ de $P$ 
sont dans $M_\Vac$. En résumé, un zéro $\alpha$ correspondant 
à une pente isolée d'un \pgn peut toujours être explicité soit 
comme un élément de $K$, soit sous une forme 
$~(a\beta+b)/(c\beta+d)~$ où $\beta$  est le zéro spécial d'un 
polynôme spécial $T$, $a,b,c,d\in V$, $~(c\beta+d)\neq 0$ et $(ad-
bc)\neq 0$.  
\end{proposition}

Notons qu'un zéro à la Hensel $\alpha$ codé par $(R,f)$ 
correspond à une pente isolée, de $(0,v(p_0))$ à $(1,v(p_1))$, 
du \pgn de $P(X)=R(X+f)$. On a donc, de manière explicite, pour tout 
zéro à la Hensel $\alpha$,  $~K[\alpha]=K[\beta]~$ où   $\beta$ 
est le zéro spécial d'un polynôme spécial.

Vu ce dernier résultat, et vus les commentaires précédents, la 
possibilité de calculer explicitement dans $(\KH,\VH)$ est ramenée 
à la proposition suivante.

\begin{proposition} 
\label{prop calc Hens} Soit $T$ un polynôme spécial, $\beta$ son 
zéro spécial dans $\KH$ et $Q\in K[X]$, alors une \dimm de 
$Q(\beta)$ peut être obtenue par des calculs uniformes dans $(K,V)$.
\end{proposition}

\begin{proof}{Preuve} 
Si $\beta=1$ il n'y a rien à faire. On supposera donc $\beta\neq 1$ 
(remarquons qu'on n'a pas en général de test permettant de savoir 
si $\beta\in K$ ou $\beta\notin K$). Appelons $\beta_1=\beta, 
\beta_2,\ldots ,\beta_d$ les zéros de $T$ dans $\Kac$ (on peut les 
supposer tous non nuls).

\sni Tout d'abord, nous donnons un algorithme qui permet de calculer 
$v(Q(\beta))\in\GtKa$, sans pour autant arriver nécessairement à 
l'expliciter comme élément de $\GtK$ (ceci sera fait, à coup 
s\^ur, dans un deuxième temps). Nous remarquons que   
$~v((1-\beta)Q(\beta))>v(Q(\beta))~$, sauf si $Q(\beta)=0$, tandis que
$~v((1-\beta_i)Q(\beta_i))=v(Q(\beta_i))~$ pour $i>1$. 

\noi Soit $~Q_1$ et $Q_2$ les deux polynômes dans $K[X]$ définis par

\snic{Q_1=\prod_{i=1,\ldots,d}(X-Q(\beta_i))~~~~~~{\rm  et}~~~~~~  
Q_2=\prod_{i=1,\ldots,d}(X-(1-\beta_i)Q(\beta_i))
}

\noi (ces deux polynômes sont faciles à calculer{\footnote{Par 
exemple si  $M$ est la matrice compagnon du poynome  $T$  ayant pour 
racines des $\beta_i$, le polynôme caractéristique de $M$ est egal au 
produit des $(X-\beta_i)$, et le polynôme caractéristique de  $Q(M)$ 
est egal au produit des $(X-Q(\beta_i))$.}}).  Il suffit d'établir 
la liste ordonnée des valuations des zéros de $Q_1$ dans $\GtKa$ 
et de la comparer à  la liste ordonnée des valuations des zéros 
de $Q_2$ dans $\GtKa$ pour tester si $Q(\beta )=0$, et dans le cas 
où $Q(\beta )\not= 0$, pour donner $v(Q(\beta))$ dans la clôture 
divisible de $\GK$ comme égal à la pente finie d'un segment bien 
déterminé du \pgn de $Q_1$. 
Si nous sommes dans le cas où $Q(\beta )\not=0$, nous allons 
maintenant  utiliser notre connaissance de $v(Q(\beta))$ en tant 
qu'élément de la clôture divisible de $\GK$ pour calculer une 
\dimm de $Q(\beta)$.

\sni Pour cela on remarque que 
$~v(\beta^mQ(\beta))=v(Q(\beta))~$ tandis que
$~v(\beta_i^mQ(\beta_i))=v(Q(\beta_i))+mv(\beta_i)>v(Q(\beta_i))~$ 
pour $i>1$ (sauf si $Q(\beta_i))=0$). Il existe  un entier $m$ tel que 
$mv(\beta_i)\neq v(Q(\beta))-v(Q(\beta_i))$ pour tous les entiers 
$i=2,\ldots ,d$. Un tel entier peut être calculé de la manière 
suivante~: il suffit que $mv(\beta_i)\neq v(Q(\beta_j))-v(Q(\beta_k))$ 
pour tous $j,k$ et tout $i>1$ (en se restreignant aux valuations 
finies). Comme le second membre ne prend qu'un nombre fini de valeurs 
et comme les  $v(\beta_i)$ ($i>1$) sont $>0$, on trouve $m$ après un 
nombre fini de tests concernant la liste des $v(\beta_i)$ ($i>1$) et 
celle des $v(Q(\beta_j))$. Pour un tel $m$ on calcule le polynôme 
$Q_3=\prod_{i=0,\ldots,d}(X-\beta_i^mQ(\beta_i))$. On sait que le 
zéro 
$\beta^mQ(\beta))$ de $Q_3$ correspond à une pente isolée de son 
\pgn et comme on connait la valeur de 
$~v(\beta^mQ(\beta))=v(Q(\beta))~$ dans $\GKac$ on sait de quelle 
pente isolée il s'agit. On utilise enfin la proposition \ref{lem 
NPH} pour obtenir une \dimm de $\beta^mQ(\beta)$ dans $K$, qui est 
aussi une \dimm de $Q(\beta)$.
\end{proof}

\section{La construction du hensélisé} 
\label{sec cons hens}
Du point de vue des mathématiques classiques, pour lesquelles 
l'existence de $(\Kac,\Vac)$ est assurée, la section 
\ref{sec calc hens}  donne une construction du hensélisé 
$(\KH,\VH)$ au sens suivant~:  $(\KH,\VH)$ est la limite inductive 
filtrante des extensions finies 
$(K[\alpha_1,\alpha_2,\ldots,\alpha_n],V_n)$ 
($V_n$ est l'unique anneau de valuation de 
$K[\alpha_1,\alpha_2,\ldots,\alpha_n]$ qui étend $V$) où chaque 
$\alpha_i$ est un  zéro à la Hensel (ou à la Newton-Hensel comme 
dans la proposition \ref{lem NPH}) sur le corps valué 
$(K[(\alpha_j)_{j<i}],V_{i-1})$.
Cette limite inductive est bien filtrante, car si on a deux extensions
$K[\alpha_1,\alpha_2,\ldots,\alpha_n]$ et
$K[\beta_1,\beta_2,\ldots,\beta_m]$, on a aussi l'extension
$$K[\alpha_1,\alpha_2,\ldots,\alpha_n][\beta_1,\beta_2,\ldots,\beta_m]
=
K[\beta_1,\beta_2,\ldots,\beta_m][\alpha_1,\alpha_2,\ldots, 
\alpha_n].$$
Enfin, il s'agit bien d'une construction car nous savons expliciter
la structure de chacune de ces extensions finies par des calculs 
uniformes qui n'utilisent que la structure de $(K,V)$ et la 
description de chaque générateur successif $\alpha_i$ par un code 
à la Hensel (ou à la Newton-Hensel).  

Pour que cette construction soit certifiée en mathématiques 
constructives, il faut démontrer constructivement que les calculs de 
la section \ref{sec calc hens} ne conduisent à aucune contradiction. 
Si cette preuve est faite, alors on peut voir  $(\KH,\VH)$ comme 
l'objet construit par ces calculs, et non plus, comme c'était le cas 
(intuitivement) dans la section \ref{sec calc hens}, comme un objet 
déjà construit a priori par une méthode abstraite. Le 
théorème \ref{th Tac} (tel qu'il est prouvé dans \cite{CLR}) 
nous donne  la preuve constructive de la cohérence des calculs dans 
la section \ref{sec calc hens}~: il suffit de vérifier que tous les 
raisonnements qui certifient ces calculs du point de vue des 
mathématiques classiques sont en fait des raisonnements qui peuvent 
être développés au sein de la théorie formelle $\Tac(K,V)$. 
Ainsi nous avons obtenu la preuve constructive du théorème 
suivant.
\begin{theorem} 
\label{th cons hens} \'Etant donné un corps valué $(K,V)$ dans 
lequel l'égalité et la divisibilité sont testables, il existe 
une  extension valuée $(\KH,\VH)$ qui contient un zéro à la 
Hensel pour tout code à la Hensel $(P,a)$ dans $~\VH[X]\times \VH$. 
Cette extension est 
unique à (K,V)--isomorphisme unique près si on réclame qu'elle 
soit engendrée par le processus de rajout de zéros à la Hensel. 
C'est une extension immédiate explicite de $(K,V)$~: tout 
élément de $\KH$ possède une \dimm dans $K$. Enfin l'égalité 
et la divisibilité sont testables dans $(\KH,\VH)$. 
\end{theorem}

Signalons que la preuve que nous venons de donner (pour l'existence du 
hensélisé d'un corps valué) est l'analogue de celles 
développées pour la clôture réelle d'un corps ordonné dans 
\cite{Del} et \cite{Lom1} plutôt que de celles données dans 
\cite{Hol},  \cite{Sand} ou~\cite{LR}. 

\ms Le raisonnement précédent peut être rendu plus agréable et 
plus intuitif si on adopte le point de vue de l'évaluation 
dynamique, tel qu'il est développé dans \cite{CLR} ou  \cite{Lom2} 
à partir de l'intuition donnée par le système de calcul formel 
D5 (cf. \cite{DDD85} et \cite{DD89}). On constate que les axiomes des 
corps valués (et ceux des corps valués algébriquement clos) 
peuvent être mis sous une forme structurellement simple, ce que nous 
appelons des règles dynamiques. Une évaluation dynamique d'une 
structure algébrique incomplétement spécifiée revient alors 
à appliquer ces règles dynamiques à la structure 
incomplètement spécifiée. Par exemple un corps valué peut 
être compris comme un corps valué algébriquement clos 
incomplètement spécifié. Les calculs uniformes de la section    
\ref{sec calc hens} se situent naturellement dans ce cadre~: ils 
fournissent une évaluation dynamique du corps valué $(K,V)$ 
lorsqu'on le considère comme un corps valué algébriquement clos 
incomplètement spécifié. Une évaluation dynamique est un 
calcul arborescent recelant parfois des ambigüités incontournables 
(dans la clôture algébrique d'un corps ``général'' par exemple, 
les zéros d'un polynôme  sont jusqu'à un certain point 
indiscernables et cela crée une telle ambigüité, qui est la source 
de l'impossibilité d'exhiber constructivement la clôture 
algébrique comme un objet global parfaitement controlé). 
L'important est cependant, pour une structure dynamique, qu'elle ne 
s'effondre pas sous le coup d'une preuve de $1=0$ par exemple. Le 
théorème 4.3 dans \cite{CLR} certifie que l'évaluation dynamique 
d'un corps valué comme corps valué algébriquement clos ne 
produit pas de contradiction. Au sein des évaluations dynamiques du 
corps valué $(K,V)$, la partie ``hensélisé'' ne recèle pas 
d'ambigüité, et c'est ce qui fait que le hensélisé peut être 
en toute généralité construit comme un objet global parfaitement 
controlé, contrairement à la clôture algébrique valuée. 

\ms Nous terminons en remarquant que la définition abstraite du 
hensélisé $\KH$ comme corps fixe d'un certain groupe 
d'automorphismes de $\Kac$ possède une version ``dynamique'' qui lui 
est classiquement équivalente via le théorème de complétude de 
Gödel, et qui constitue donc une version constructivement 
satisfaisante du théorème classique correspondant. Le théorème 
classique dit que la définition abstraite ``par le haut'' de $\KH$ 
comme corps fixe d'un certain groupe coïncide avec une 
définition terre à terre ``par le bas'' (le plus petit corps stable 
par rajout de racines à la Hensel). Le théorème dynamique 
correspondant en caractéristique nulle est le suivant (nous ne 
savons pas pour le moment en donner une preuve constructive)~:
\begin{theorem} 
\label{th hensel 2} Considérons une présentation $(K,V)$ de corps 
valué de caractéristique nulle et évaluons la dynamiquement en 
tant que corps algébriquement clos valué.  Soit $t$ un terme 
apparaissant dans un tel arbre. Alors les deux propriétés 
suivantes sont équivalentes~:
\begin{itemize}
\item $t$ est égal (au sens de l'évaluation dynamique) à un 
élément de $K[\xi_1,\ldots,\xi_n]$ où les $\xi_j$ sont des 
zéros à la Hensel rajoutés en cascade les uns après les 
autres.
\item  si l'on introduit un symbole de fonction $\varphi$ soumis aux 
règles dynamiques spécifiant que $\varphi$ est un automorphisme de 
$\Kac$ qui fixe $K$ point par point et $\Vac$ globalement, alors $t$ 
est égal à $\varphi(t)$  (au sens de l'évaluation dynamique).

\end{itemize}

\end{theorem}



\tableofcontents
\end{document}